\documentclass[12pt]{amsart}
\usepackage{amsmath,amssymb}
\usepackage{amsfonts}
\usepackage{amsthm}
\usepackage{latexsym}
\usepackage{graphicx}

\def\p{\partial}

\def\R{\mathbb{R}}

\def\vv<#1>{\langle#1\rangle}
\def\ol{\overline}

\def\s1{{\mathbb{S}^1}}
\def\e{{\epsilon}}

\def\XXint#1#2{\setbox0=\hbox{$#1{#2}{\int}$}{#2}\kern-.5\wd0 }

\def\XXint#1#2#3{{\setbox0=\hbox{$#1{#2#3}{\int}$}
     \vcenter{\hbox{$#2#3$}}\kern-.5\wd0}}

\def\vv<#1>{{\left\langle#1\right\rangle}}

\newtheorem{thm}{Theorem}[section]

\newtheorem{lem}{Lemma}[section]

\newtheorem{cor}{Corollary}[section]
\theoremstyle{definition}

\theoremstyle{remark}

\newtheorem{rem}{Remark}[section]

\numberwithin{equation}{section}

\begin{document}
\title{Chen-Nester-Tung quasi-local energy and Wang-Yau quasi-local mass}
\author{Jian-Liang Liu$^1$}
\address{Department of Mathematics, Shantou University, Shantou, Guangdong, 515063, China}
\email{liujl@stu.edu.cn}
\thanks{$^1$Research partially supported by the China Postdoctoral Science Foundation 2016M602497 and NSFC 61601275. }
\author{Chengjie Yu$^2$}
\address{Department of Mathematics, Shantou University, Shantou, Guangdong, 515063, China}
\email{cjyu@stu.edu.cn}
\thanks{$^2$Research partially supported by a supporting project from the Department of Education of Guangdong Province with contract no. Yq2013073, and NSFC 11571215.}
\renewcommand{\subjclassname}{%
  \textup{2010} Mathematics Subject Classification}
\date{}
\begin{abstract}
In this paper, we show that the Chen-Nester-Tung (CNT) quasi-local energy with 4D isometric matching references is closely  related to the Wang-Yau (WY) quasi-local energy.  As a particular example, we compute the second variation of the CNT quasi-local energy for axially symmetric Kerr-like spacetimes with axially symmetric embeddings at the obvious critical point $(0,0)$ and find that it is a saddle critical point in most of the cases. Also, as a byproduct, we generalize a previous result about the coincidence of the CNT quasi-local energy and Brown-York mass for axially symmetric Kerr-like spacetimes by Tam and the first author \cite{LT} to general spacetimes.
\end{abstract}
\maketitle\markboth{Liu \& Yu}{CNT \& WY}
\section{Introduction}
The problem of defining the energy of gravitating systems is its localization. There is no proper local description of gravitational energy. This was shown by Noether a hundred years ago \cite{Noether}. This phenomenon is physically understood in terms of the equivalence principle, and it is simply a basic fact about the local flatness of Riemannian geometry which implies that one can always find a coordinate system such that at any selected point the connection coefficients vanish.  The connection coefficients are analogous to the gravitational force.

In the early days, many efforts were made to construct some gravitational energy density, they led to various pseudotensors. Because of the fundamental property of gravity, two ambiguities arise: (i) there are many possible expressions and (ii) they are non-covariant (i.e. coordinate dependent) \cite{CNT1}.  Later, the idea of \emph{quasi-local} was used, and there are several proposals on defining quasi-local quantities \cite{Sz}. Chang et.al found that different pseudotensor expressions are related to different boundary conditions associated with the Hamiltonian boundary expression \cite{CNC}.  One may treat gravitational energy from the Hamiltonian point of view as well as from pseudotensor. The advantage of using the Hamiltonian formalism is that the Noether conserved current is the Hamiltonian density, which is the canonical generator of a local spacetime translation \cite{CNT1}, and it implies that different conserved quantities correspond to different displacement vector fields \cite{CNT4}.

Although the Hamiltonian point of view gives a clear understanding of the conservation and symmetry, the fundamental ambiguities are still there: (i) many possible boundary conditions of the Hamiltonian boundary term (i.e.\ Hamiltonian boundary expressions) and (ii) the choice of the reference (which is related to the coordinate choice of the pseudotensor expression).

In this article, we analyze the expression favored by CNT (\cite{CNT4}, eq.(57) which is related to the fixed coframe $\vartheta^{\alpha}$ on the boundary as the boundary condition):
\begin{equation}
E(N,\Omega)=\int_{\Omega}\mathcal{H}(N)=\oint_{S}\mathfrak{B}(N),
\end{equation}
where $\Omega$ is a spacelike region with a closed 2-boundary $S$,
\begin{equation}
\mathfrak{B}(N)=\frac{1}{2\kappa}(\Delta\omega^{a}{}_{b}\wedge i_{N}\eta_{a}{}^{b}+\ol{D}_{b}N^{a}\Delta\eta_{a}{}^{b}),
\end{equation}
$\eta_{a}{}^{b}$ will be defined below. Here $\kappa=8\pi G/c^4$, $G$ is the Newtonian gravitational constant and usually we take $c=G=1$; $\Delta \alpha:=\alpha-\bar{\alpha}$ is the difference of the variables in the physical spactime and the reference spacetime.

The CNT proposal developed a manifestly 4D covariant Hamiltonian formalism which can be applied on a wide class of geometric gravity theories, including GR, and it does not necessary depend on 3+1 decomposition. It covers not only energy but also other quasi-local quantities. Here we focus only on the quasi-local energy which corresponds to the timelike vector field $N$ on $S$. The vector field $N$ is called a displacement vector field in \cite{CNT4}. It is required to be identified to a timelike Killing vector field in the reference space-time.

Let $(M^4,g)$ be a spacetime (oriented and time-oriented) which is considered as the physical spacetime and $(\ol M^4,\ol g)$ be another spacetime (oriented and time-oriented) which is considered as the reference spacetime. Let $S^2$ be a closed spacelike surface in $M$ and $N$ be a future-directed timelike vector field on $S$. We call a smooth embedding $\varphi:U\to \ol M$ of an open neighborhood $U$ of $S$ into $\ol M$ that preserves the orientation and the time orientation a reference. Then, the CNT quasi-local energy of $S$ with respect to $N$ and the reference $\varphi$ is defined as
\begin{equation}
E(S,N,\varphi)=\frac{1}{2\kappa}\int_{S}\iota^*[(\omega^a{}_{b}-\ol\omega^a{}_b)\wedge i_N\eta_{a}{}^b+\ol{D}_{b}N^{a}(\eta_{a}{}^{b}-\ol\eta_{a}{}^{b})].\label{ECNT}
\end{equation}
Here, $\omega^a{}_b$ and $\ol \omega^a{}_b$ are the connection forms of the Levi-Civita connections for $g$ and $\varphi^*\ol g$ respectively, $\iota:S\to M$ is the natural inclusion map, and
\begin{equation}
\eta_a{}^b=\frac{1}{2}\sqrt{-\det g}g^{b\beta}\epsilon_{a\beta\mu\nu}dx^\mu\wedge dx^\nu,
\end{equation}
the covariant derivative $\ol D_{a}$ and the 2-form $\ol\eta_{a}^{\ b}$ correspond to $\varphi^*\ol g$.

Usually, the reference spacetime is chosen to be the Minkowski spacetime, dS spacetime or AdS spacetime. The main difficulty for the CNT quasi-local energy comes from the choice of canonical references so that desired properties are satisfied. In this paper, we follow the strategy of choosing reference by 4D isometric matching\footnote{The first author would like to thank Dr.\ Szabados for helpful discussions on this topic when he visited NCU at 2011. At that time, Wu et.\ al began to use CNT expression investigate the spherical symmetric cases by matching the 4-coframes \cite[\S 4 p.2411]{Wu}.} mentioned in \cite{NCLS,SCLN2,SCLN3}, and analyze the critical value of quasi-local energy. This method of finding reference is to determine a coordinate transformation such that the whole 10 metric components of the physical spacetime and the reference spacetime are identical right on the quasi-local 2-surface. It can be realized based on the 2-surface isometric embedding into the reference spacetime. Epp defined the ``invariant quasilocal energy'' by considering the 2-surface isometric embedding into  the Minkowski spacetime \cite{Epp}. Wang and Yau consider 2-surface isometric embedding into Minkowski spacetime and proved the positivity of quasi-local mass by fixing the canonical gauge \cite{WaYa,W}. For a more complete survey of the topic, see \cite{Sz}.

A 4D isometric matching reference $\varphi$ is a reference satisfying
\begin{equation}
\varphi^*\ol g=g \mbox{ on $S$}.
\end{equation}
A basic problem about the existence of 4D isometric matching references arises here. By a rather standard argument using the exponential map, we can show that any isometric embedding of $S$ into $\ol M$ can be extended to a 4D isometric matching reference (see Lemma \ref{lem-ref}). When $\ol M$ is the Minkowski spactime, isometric embedding of $S$ into $\ol M$ was discussed in \cite{WaYa}, \cite{Epp}(with a further restriction) and \cite{Br,BLY,Sz} (into the light cone).

Let $\varphi$ be a 4D isometric matching reference, since $\eta_a{}^b$ depnds only on the metric, so the second term in \eqref{ECNT} vanishes and the CNT quasi-local energy \eqref{ECNT} with respect to $\varphi$ becomes\footnote{In the coordinate system such that the reference connection vanishes, it is reduced to the Freud superpotential \cite{Freud}.}
\begin{equation}\label{eqn-CNT-E-1}
E(S,N,\varphi)=\frac{1}{8\pi}\int_S \iota^*[(\omega^a{}_b-\ol\omega^a{}_b)\wedge i_N\eta_{a}{}^b].
\end{equation}
Although a 4D isometric matching reference must be defined in a neighborhood of $S$ by definition, it is not hard to see that the CNT quasi-local energy depends only on the 1-jet of a 4D isometric matching reference $\varphi$ on $S$ (see Lemma \ref{lem-1-jet}). That is to say, if $\varphi_1$ and $\varphi_2$ are two 4D isometric matching references such that $\varphi_1=\varphi_2$ and $d\varphi_1=d\varphi_2$ on $S$, then
\begin{equation}
E(S,N,\varphi_1)=E(S,N,\varphi_2).
\end{equation}

By the fact that CNT quasi-local energy depends only on the 1-jet of the 4D isometric matching reference and the extension of an isometric embedding to a 4D isometric matching reference (see Lemma \ref{lem-ref}), we can simply consider a 4D isometric matching reference as a pair $(\varphi, \psi)$ where $\varphi:S\to \ol M$ is an isometric embedding and $\psi:T_{S}^\perp{M}\to \varphi^*T_{\ol S}^\perp\ol M$ is a linear isometry of vector bundles. Here $T_S^{\perp}M$ and $T_{\ol S}^\perp \ol M$ mean the normal bundle of $S$ and $\ol S$ respectively. $\varphi$ can be considered as the embedding freedom and $\psi$ as the boost freedom of the observer in \cite{LT}.

In this paper, we find that the CNT quasi-local energy with respect to a 4D isometric matching reference is closely related to the WY mass. Indeed, a saddle critical value of the CNT quasi-local energy turns out to agree with the WY energy. As a byproduct, we generalize a previous result of Tam and the first author \cite{LT} to a much more general setting. Our result is as follows:
\begin{thm}\label{thm-CNT-BY}
Let $(M^4,g)$ and $(\ol M^4,\ol g)$ be two oriented and time-oriented spacetimes, $S^2$ be an oriented closed spacelike surface in $M^4$ and $N$ be a  future-directed timelike vector field on $S$. Let $\varphi$ be a 4D isometric matching reference. Then,
\begin{equation}\label{eqn-CNT}
\begin{split}
&E(S,N,\varphi)\\
=&\frac{1}{8\pi}\int_{\ol S}\left(-\|\ol N^\perp\|\vv<\ol H,\ol X>+\frac{1}{\|\ol N^\perp\|}\vv<\ol\nabla_{\ol N^\top}\ol X,\ol N^\perp>\right)dV_{\ol S}\\
&-\frac{1}{8\pi}\int_S\left(-\|N^\perp\|\vv<H,X>+\frac{1}{\|N^\perp\|}\vv<\nabla_{N^{\top}}X,N^\perp>\right)dV_S
\end{split}
\end{equation}
where $H$ and $\ol H$ are the mean curvature vectors of $S$ and $\ol S=\varphi(S)$ in $M$ and $\ol M$ respectively. Here $N^\perp$ is the orthogonal projection of $N$ into $T_S^\perp M$, $N^\top=N-N^\perp$, and $X$ is the unit vector field on $S$ such that $X\in T_S^\perp M$ and $X\perp N$. We also assume that the composition of $N^\perp, X$ and the orientation of $S$ is the same as the orientation of $M$. Moreover, $\ol X=\varphi_*X$, $\ol N^\perp=\varphi_*N^\perp$ and $\ol N^\top=\varphi_*N^\top$. In particular, if $N$ is orthogonal to $S$, then
\begin{equation}
E(S,N,\varphi)=\frac{1}{8\pi}\int_{S}\|N\|\left(-\vv<\ol H,\ol X>+\vv<H,X>\right)dV_S.
\end{equation}
\end{thm}
\begin{rem}
The right hand side of \eqref{eqn-CNT} is similar to equation (6) in Wang and Yau \cite{WYPRL}.
\end{rem}
For $S$ enclosing a spacelike domain $\Omega$ with $S$ embedded into $\R^3$, the Brown-York mass \cite{BY} was defined as
\begin{equation}
\mathfrak{m}_{\mathrm{BY}}(\Omega)=\frac{1}{8\pi}\int_{S}(k_0-k)dV_S,
\end{equation}
where $k$ and $k_0$ are the mean curvature (with respect to unit outward normal) of $S$ and the embedding of $S$ respectively. Comparing this to the result above, we have the following corollary:
\begin{cor}\label{cor-CNT-BY}
Let $M,S,N$ be the same as above and $(\ol M,\ol g)$ be the Minkowski spacetime. Moreover, suppose that $S$ encloses a space-like domain $\Omega$ and $N$ is also orthogonal to $\Omega$. Then, for any 4D isometric matching reference $\varphi$ such that
\begin{enumerate}
\item $\varphi_*N=\frac{\p}{\p T}$ and
\item  $\varphi(S)\subset\R^3$,
\end{enumerate}
we have $E(S,N,\varphi)=\mathfrak{m}_{\mathrm{BY}}(\Omega)$. Here, the natural coordinate of $\ol M$ is written as $(T,X,Y,Z)$.
\end{cor}
\begin{rem}
The orientation of $S$ is chosen so that $N$, the outward normal of $S$ and the orientation of $S$ form the orientation of $M$.
\end{rem}

Recall that the dS spacetime and AdS spacetime are $\R\times \hat M$ with $\hat M$ be the sphere and hyperbolic space respectively, equipped with the Lorentz metric:
\begin{equation}
g=-V^2dT^2+\hat g
\end{equation}
where $\hat g$ is the standard metric on $\hat M$ and $V$ is the static potential on $\hat M$ (see \cite{CH}). Note that $\frac{\p}{\p T}$ is a future directed time-like Killing vector field on the dS spacetime and AdS spacetime with length $|V|$. So, we have the following corollary when the reference is chosen to be the dS spacetime  or AdS spacetime corresponding to Corollary \ref{cor-CNT-BY} where the reference is chosen to be the Minkowski spacetime.

\begin{cor}\label{cor-CNT-BY_dSAdS}
Let $M,S,N$ be the same as above and $(\ol M,\ol g)$ be the dS or AdS spacetime. Moreover, suppose that $S$ encloses a space-like domain $\Omega$ and $N$ is also orthogonal to $\Omega$. Then, for any 4D isometric matching reference $\varphi$ such that
\begin{enumerate}
\item $\varphi_*N=\frac{\p}{\p T}$ and
\item  $\varphi(S)\subset \hat M$ ,
\end{enumerate}
we have
\begin{equation}\label{eqn-CNT-ds-ads}
E(S,N,\varphi)=\frac{1}{8\pi}\int_{S}|V|(k_0-k)dV_S.
\end{equation}
\end{cor}
Note that the expression \eqref{eqn-CNT-ds-ads} was first studied in \cite{ST,WY} which gives a substitution for Brown-York mass with dS or AdS references. Quasi-local energy with dS and AdS references was also studied in \cite{CWY2}.

Let $\varphi_0:S\to \R^{1,3}$ be an isometric embedding and $\tau$ be the time component of the embedding and suppose the mean curvature vector $H$ of $S$ in $M$ is spacelike. Recall that the Wang-Yau \cite{WaYa} quasi-local energy $E_{\mathrm{WY}}(S,\tau)$ is defined as
\begin{equation}\label{eqn-WY}
\begin{split}
E_{\mathrm{WY}}(S,\tau)=&\frac{1}{8\pi}\int_{\ol S}\left(-\sqrt{1+\|\nabla\tau\|^2}\vv<\ol H,\ol e_1>+\vv<\ol \nabla_{-\nabla\tau}\ol e_1,\ol e_0>\right)dV_{\ol S}\\
&-\frac{1}{8\pi}\int_{S}\left(-\sqrt{1+\|\nabla\tau\|^2}\vv<H, e_1>+\vv< \nabla_{-\nabla\tau}e_1,e_0>\right)dV_S.\\
\end{split}
\end{equation}
Here $e_0$ is a future-directed time-like vector such that
\begin{equation}\label{eqn-e-0}
\vv<H,e_0>=-\frac{\Delta\tau}{\sqrt{1+\|\nabla\tau\|^2}},
\end{equation}
$e_1$ is orthogonal to $e_0$ and $S$, and pointed outside if $S$ encloses a domain $\Omega$. $\ol e_1$ is pointing outside and orthogonal to $\ol S$ and $\frac{\p}{\p T}$. $\ol e_0$ is a future-directed time-like vector that is orthogonal to $\ol S$ and $\ol e_1$. Then, the WY mass of $S$ is defined as \begin{equation}
\mathfrak{m}_{\mathrm{WY}}(S)=\inf_{\mbox{$\tau$ admissible}}E_{\mathrm{WY}}(S,\tau).
\end{equation}
Here $\nabla \tau$ and $\Delta\tau$ mean the gradient and Laplacian of $\tau$ with respect to the induced metric on $S$. For the definition of admissible, see \cite{WaYa}.

Let
\begin{equation}\label{eqn-N-0-1}
N_0=\sqrt{1+\|\nabla \tau\|^2}e_0-\nabla\tau
\end{equation}
and
$\varphi$ be a 4D isometric matching extension of $\varphi_0$ (see Lemma \ref{lem-ref}) such that
\begin{equation}\label{eqn-N-0}
\varphi_*N_0=\frac{\p}{\p T}.
\end{equation}
By comparing \eqref{eqn-CNT} with \eqref{eqn-WY}, one can see that
\begin{equation} \label{eqn-CNT-WY}
E(S,N_0,\varphi)=E_{\mathrm{WY}}(S,\tau).
\end{equation}
Moreover, by the computation in \cite[p.925]{WaYa},
\begin{equation}
E(S,N_0,\varphi)=\max_{\psi_*N=\frac{\p}{\p T},\psi|_S=\varphi_0 }E(S,N,\psi).
\end{equation}
This means that the WY quasi-local mass value can be obtained from a min-max procedure from the CNT quasi-local energy. Therefore, if the WY quasi-local mass is achieved by an isometric embedding $\varphi_0$, then the corresponding pair $(N_0,\varphi)$ is actually a saddle critical point of the CNT quasi-local energy in the space of references:
\begin{equation*}
\mathcal R_0=\{(N,\psi)\ |\ \mbox{$\psi$: a 4D isometric matching reference with $\psi_*N=\frac{\p}{\p T}$} \}.
\end{equation*}

Consider the physical spacetime $(M,g)$ being axially symmetric and Kerr-like:
\begin{equation}\label{eqn-kerr-like}
g=Fdt^2+2Gdtd\phi+Hd\phi^2+R^2dr^2+\Sigma^2d\theta^2,
\end{equation}
where the components $F,G,H,R,\Sigma$ are functions of $r$, $\theta$ only. Let
$$\Omega=\{(t,r,\theta,\phi)\ |\ t=t_0,r\leq r_0 \}$$ and  $$S=\{(t,r,\theta,\phi)\ |\ t=t_0,r=r_0 \}.$$
Suppose that the 4D isometric matching reference $\varphi$ is axially symmetric:
\begin{equation}\label{eqn-axis-embedding}
\left\{\begin{array}{l}T=T(t,r,\theta)\\
X=\rho(t,r,\theta)\cos(\phi+\Phi(t,r,\theta))\\
Y=\rho(t,r,\theta)\sin(\phi+\Phi(t,r,\theta))\\
Z=Z(t,r,\theta).
\end{array}\right.
\end{equation}
Then, the 4D isometric matching equation is indeed explicitly solvable (See \cite{SCLN2}). Set
\begin{equation}\label{eqn-x-y}
\left\{\begin{array}{l}x(\theta)=T_r(t_0,r_0,\theta)\\
y(\theta)=T_\theta(t_0,r_0,\theta)
\end{array}\right.
\end{equation}
and
\begin{equation}
N=(\varphi^{-1})_*\frac{\p}{\p T}.
\end{equation}
Then,
\begin{equation}\label{eqn-CNT-x-y}
E(x,y):=E(S,N, \varphi)=\frac{1}{4}\int_0^{\pi}\mathfrak{B}(x,y)d\theta
\end{equation}
with
\begin{equation}\label{eqn-CNT-B}
\begin{split}
\mathfrak{B}(x,y)=&-\frac{\alpha(H\Sigma^2)_r}{2\sqrt H R^2\Sigma^2}-\sqrt H\left(\frac{H_{\theta\theta}-2l}{\beta}+\frac{R_\theta xy}{R\alpha}-\frac{xy^3\beta+H_\theta\alpha\Sigma^2}{l\alpha\beta\Sigma}\Sigma_\theta\right)\\
&+\frac{\sqrt Hyx_\theta}{\alpha}+\frac{\sqrt Hy(H_\theta\alpha-xy\beta)}{l\alpha\beta}y_\theta,
\end{split}
\end{equation}
where
\begin{equation}\label{eqn-a-b-l}
\left\{\begin{array}{l}\alpha=\sqrt{x^2\Sigma^2+R^2l}\\
\beta=\sqrt{-H_\theta^2+4Hl}\\
l=y^2+\Sigma^2.
\end{array}\right.
\end{equation}
Here we call $x$ the boost freedom and $y$ the embedding freedom. Note that
$$H(0)=H(\pi)=G(0)=G(\pi)=0$$
and
$$x'(0)=x'(\pi)=y(0)=y(\pi)=0$$
by the smoothness of the metric $g$ and the function $T$ (as a smooth function on $S$) (See \eqref{eqn-kerr-like} and \eqref{eqn-axis-embedding}).

The Euler-Lagrange equation of $E(x,y)$ is
\begin{equation}\label{eqn-critical}
\left\{
\begin{split}
y_\theta=&-\frac{(\Sigma^2H)_r}{2HR^2}x-\frac{\Sigma H_\theta-2H\Sigma_\theta}{2H\Sigma}y\\
x_\theta=&\frac{R_\theta}{R}x+\left(\frac{(\Sigma^2H)_r}{2H\Sigma^2}-\frac{\alpha\beta+xyH_\theta}{2Hl}\right)y.
\end{split}\right.
\end{equation}
\eqref{eqn-critical} has an obvious solution $x\equiv y\equiv 0$. In \cite{LT}, Tam and the first author considered the solutions of \eqref{eqn-critical} in the case of the Minkowski, Schwarschild and Kerr spacetimes. They also compared the CNT quasi-local energy with the Brown-York mass when $x=y=0$. Indeed, they showed that
\begin{equation}
E(0,0)=\mathfrak{m}_{\mathrm{BY}}(\Omega).
\end{equation}
It is not difficult to see that this is a special case of Corollary \ref{cor-CNT-BY}.

Motivated by the relation of WY mass and CNT quasi-local energy, we compute the second variation of $E(x,y)$ at the obvious critical point $x=y=0$ and find that, for most of the cases, $(x,y)=(0,0)$ is a saddle point (see Theorem \ref{thm-saddle} ). In particular, this is true for the Minkowski, Schwarschild and Kerr spacetimes.

By direct computation (see the Appendix), the first equation of \eqref{eqn-critical} corresponds to
\begin{equation}
N_0=(\varphi^{-1})_*\frac{\p}{\p T}.
\end{equation}
Let
\begin{equation}
E(y):=E(x,y)
\end{equation}
with $x$ decided by $y$ from the first equation of \eqref{eqn-critical}. Then, by the uniqueness of isometric embedding into $\R^3$ and  the relation \eqref{eqn-CNT-WY}, we know that
\begin{equation}
E(y)=E_{\mathrm{WY}}(S, \tau),
\end{equation}
where $\tau$ depends only on $\theta$ and $y=\frac{d\tau}{d\theta}$.

In \cite{CWY,MT,MTX}, the authors considered minimizing of properties of the critical points of the WY quasi-local energy. By their results, it is clear that $y=0$ is a local minimum of $E(y)$ for the Schwarschild spacetime when $r>2m$ and for the Kerr spacetime when $r$ is large enough. This implies that $(0,0)$ is a saddle critical point of $E(x,y)$ for the Schwarschild spacetime when $r>2m$ and for the Kerr spacetime when $r$ is large enough.

Furthermore, in \cite{CWY}, under some curvature assumptions, when the induced metric on $S$ is axially symmetric,  Chen, Wang and Yau showed that if $\tau=0$ is a critical point of the WY quasi-local energy, then $\tau=0$ is a global minimum among all axially symmetric $\tau$. This implies that
\begin{equation}
E(y)\geq E(0,0)
\end{equation}
for the Schwarzschild spacetime when $r>2m$ and for the Kerr spacetime when $r$ is large enough. It is very likely that
\begin{equation}
\begin{split}
\mathfrak{m}_{\mathrm{WY}}(S)&=E(0,0)\\
&=\mathfrak{m}_{\mathrm{BY}}(S)\\
&=\frac{1}{4}\int_0^\pi\left(-\frac{(\sqrt H\Sigma)_r}{ R}+\frac{\Sigma\left(1-\frac{H_{\theta\theta}}{2\Sigma^2}\right)}{\sqrt{1-\frac{H_\theta^2}{4H\Sigma^2}}}+\frac{H_\theta\Sigma_\theta}{2\Sigma^2\sqrt{1-\frac{H_\theta^2}{4H\Sigma^2}}}\right)d\theta
\end{split}
\end{equation}
for the Kerr spacetime.

This paper is organized as follows. In Section 2, we prove  Theorem \ref{thm-CNT-BY}. In Section 3, we compute the second variation of $E(x,y)$ at the obvious critical point $(x,y)=(0,0)$ and show that it is a saddle point for most of the cases including the Minkowski, Schwarschild  and Kerr spacetimes.

{\bf Acknowledgements.} The authors would like to thank Professor Nester for carefully reading the manuscript of this paper, many helpful suggestions and sharing ideas, Professor Tam for helpful suggestions, and Professor Mu-Tao Wang for his comments that help to clarify our understanding.
\section{CNT quasi-local energy and Wang-Yau mass}
We first prove that any isometric embedding can be extended to a 4D isometric reference. The argument is rather standard and it may be trivial for experts. However, for completeness of the paper, we also include the proof.
\begin{lem}\label{lem-ref}
Let $(M^4,g)$ and $(\ol M^4,\ol g)$ be two oriented and time-oriented spacetimes, $S^2$ be an oriented closed spacelike surface in $M$, and $\varphi_0:S\to \ol M$ be an isometric embedding. Moreover, let $N$ and $\ol N$ be unit future-directed timelike normal vector fields on $S$ and $\ol S:=\varphi_0(S)$ respectively, and $X$ and $\ol X$ be unit normal vector fields on $S$ and $\ol S$ that are also orthogonal to $N$ and $\ol N$ respectively. We also assume that the compositions of $N,X$, orientation of $S$ and $\ol N$, $\ol X$, orientation of $\ol S$ (induced from $S$ by $\varphi_0$) are the same as the orientations of $M$ and $\ol M$ respectively. Then, there is a 4D isometric matching reference $\varphi$ such that
\begin{enumerate}
\item $\varphi|_{S}=\varphi_0$;
\item $\varphi_* N=\ol N$ on $\ol S$;
\item $\varphi_* X=\ol X$ on $\ol S$.
\end{enumerate}
\end{lem}
\begin{proof}
Define the reference $\varphi$ as
\begin{equation}\label{eqn-exp}
\varphi(\exp_p(rN_p+sX_p))=\exp_{\ol p}(r\ol N_{\ol p}+s\ol X_{\ol p})
\end{equation}
for any $p\in S$ and $r,s\in (-\epsilon,\epsilon)$ where $\epsilon>0$ is small enough and $\ol p=\varphi_0(p)$. Let $Y_p,Z_p$ be an orthornormal basis of $T_pS$ and
 $$\ol Y_{\ol p}={\varphi_0}_{*p}Y_p,\ \ol Z_{\ol p}={\varphi_0}_{*p}Z_p.$$
Then, $(N_p,X_p,Y_p,Z_p)$ and $(\ol N_{\ol p}, \ol X_{\ol p}, \ol Y_{\ol p}, \ol Z_{\ol p})$ are orthnormal bases for $T_pM$ and $T_{\ol p}\ol M$ respectively. By \eqref{eqn-exp}, it is clear that $\varphi|_S=\varphi_0$ and
\begin{equation}
\varphi_{*p}(N_p)=\ol N_{\ol p},\ \varphi_{*p}(X_p)=\ol X_{\ol p},\ \varphi_{*p}(Y_p)=\ol Y_{\ol p},\ \varphi_{*p}(Z_p)=\ol Z_{\ol p}.
\end{equation}
This means that $\varphi^*\ol g=g$ on $S$.
\end{proof}
Next, we prove that the CNT quasi-local energy with respect to 4D isometric matching references depends only on the 1-jet of the reference.
\begin{lem}\label{lem-1-jet}
 Let $(M,g),(\ol M,\ol g)$, $S$ and $N$ be the same as in the Lemma \ref{lem-ref}.
Let $\varphi_1$ and $\varphi_2$ be two 4D isometric matching references such that
$$\varphi_1=\varphi_2\ \mbox{and}\ d\varphi_1=d\varphi_2$$
on $S$. Then,
\begin{equation}
E(S,N,\varphi_1)=E(S,N,\varphi_2).
\end{equation}
\end{lem}
\begin{proof}
Let $e_0,e_1,e_2,e_3$ be a local orthonormal frame of $(M,g)$ such that $e_2,e_3$ are tangential to $S$ and let $\omega^0,\omega^1,\omega^2,\omega^3$ be its dual frame. By \eqref{eqn-CNT-E-1}, we only need to verify that,
$$\iota^*\ol \omega^a{}_b=\iota^*\tilde\omega^a{}_b,$$
where $\ol\omega$ and $\tilde \omega$ are the connection forms of $\varphi_1^*\ol g$ and $\varphi_2^*\ol g$ respectively. Indeed,
\begin{equation}
\begin{split}
&\iota^*\ol \omega^a{}_b\\
=&\ol\Gamma_{2b}^a\omega^2+\ol\Gamma_{3b}^a\omega^3\\
=&\vv<\ol\nabla_{e_2}e_b,\eta_{aa}e_a>_{\varphi_1^*\ol g}\omega^2+\vv<\ol\nabla_{e_3}e_b,\eta_{aa}e_a>_{\varphi_1^*\ol g}\omega^3\\
=&\vv<\ol\nabla_{d\varphi_1(e_2)}d\varphi_1(e_b),\eta_{aa}d\varphi_1(e_a)>_{\ol g}\circ\varphi_1\omega^2\\
&+\vv<\ol\nabla_{d\varphi_1(e_3)}d\varphi_1(e_b),\eta_{aa}d\varphi_1(e_a)>_{\ol g}\circ\varphi_1\omega^3\\
=&\vv<\ol\nabla_{d\varphi_2(e_2)}d\varphi_2(e_b),\eta_{aa}d\varphi_2(e_a)>_{\ol g}\circ\varphi_2\omega^2\\
&+\vv<\ol\nabla_{d\varphi_2(e_3)}d\varphi_2(e_b),\eta_{aa}d\varphi_2(e_a)>_{\ol g}\circ\varphi_2\omega^3\\
=&\iota^*\tilde \omega^a{}_b.
\end{split}
\end{equation}

\end{proof}

Finally, we come to prove Theorem \ref{thm-CNT-BY}.
\begin{proof}[Proof of Theorem \ref{thm-CNT-BY}]Let $e_0=\frac{N^\perp}{\|N^\perp\|}$, and $e_1=X$. Let $e_2,e_3$ be a local orthonormal frame of $S$ and extend $e_0,e_1,e_2,e_3$ to a local orthonormal frame of $M$. Let $\omega^0,\omega^1,\omega^2,\omega^3$ be the dual frame of $e_0,e_1,e_2,e_3$. Since $N\perp e_1$, suppose that
\begin{equation}
N=N^0e_0+N^2e_2+N^3e_3.
\end{equation}
It is clear that $N^0=\|N^\perp\|$ and $N^\top=N^2e_3+N^3e_3$.
Then
\begin{equation}\label{eqn-eta}
\begin{split}
&\iota^*i_N\eta_{a}{}^b\\
=&N^{\mu}\eta^{b\beta}\epsilon_{a\beta\mu2}\omega^2+N^{\mu}\eta^{b\beta}\epsilon_{a\beta\mu3}\omega^3\\
=&N^0(\e_{ab02}\omega^2+\e_{ab03}\omega^3)+\eta^{bb}\e_{ab23}(-N^3\omega^2+N^2\omega^3).
\end{split}
\end{equation}
Note that
\begin{equation}
\begin{split}
&\frac{1}{16\pi}\int_{S}\iota^*(\omega^a{}_b-\ol\omega^a{}_b)\wedge N^0(\e_{ab02}\omega^2+\e_{ab03}\omega^3)\\
=&\frac{1}{16\pi}\int_{S}N^0[(\Gamma_{2b}^a-\ol\Gamma_{2b}^a)\omega^2+(\Gamma_{3b}^a-\ol\Gamma_{3b}^a)\omega^3]\wedge (\epsilon_{ab02}\omega^2+\epsilon_{ab03}\omega^3)\\
=&\frac{1}{16\pi}\int_{S}N^0(\Gamma_{22}^1+\Gamma_{33}^1-\ol \Gamma_{22}^{1}-\ol\Gamma_{33}^{1})\omega^2\wedge\omega^3\\
&-\frac{1}{16\pi}\int_{S}N^0(\Gamma_{21}^2+\Gamma_{31}^3-\ol \Gamma_{21}^{2}-\ol\Gamma_{31}^{3})\omega^2\wedge\omega^3,\\
\end{split}
\end{equation}
and on $S$,
\begin{equation}\label{eqn-Gamma-22}
\Gamma_{22}^1+\Gamma_{33}^1=\vv<\nabla_{e_2}e_2+\nabla_{e_3}e_3,e_1>=\vv<H,X>,
\end{equation}
\begin{equation}\label{eqn-Gamma-21}
\Gamma_{21}^2+\Gamma_{31}^{3}=\vv<\nabla_{e_2}e_1,e_2>+\vv<\nabla_{e_3}e_1,e_3>=-\vv<H,X>.
\end{equation}
Moreover, on $S$,
\begin{equation}\label{eqn-o-Gamma-22}
\begin{split}
&\ol \Gamma_{22}^{1}+\ol\Gamma_{33}^{1}\\
=&\vv<\ol\nabla_{e_2}e_2,e_1>_{\varphi^*\ol g}+\vv<\ol\nabla_{e_3}e_3,e_1>_{\varphi^*\ol g}\\
=&\vv<\ol \nabla_{\varphi_*e_2}\varphi_*e_2,\varphi_*e_1>_{\ol g}+\vv<\ol \nabla_{\varphi_*e_3}\varphi_*e_3,\varphi_*e_1>_{\ol g}\\
=&\vv<\ol H,\ol X>_{\ol g}
\end{split}
\end{equation}
and
\begin{equation}\label{eqn-o-Gamma-21}
\begin{split}
&\ol \Gamma_{21}^{2}+\ol\Gamma_{31}^{3}\\
=&\vv<\ol\nabla_{e_2}e_1,e_2>_{\varphi^*\ol g}+\vv<\ol\nabla_{e_3}e_1,e_3>_{\varphi^*\ol g}\\
=&\vv<\ol \nabla_{\varphi_*e_2}\varphi_*e_1,\varphi_*e_2>_{\ol g}+\vv<\ol \nabla_{\varphi_*e_3}\varphi_*e_1,\varphi_*e_3>_{\ol g}\\
=&-\vv<\ol H,\ol X>_{\ol g}.
\end{split}
\end{equation}
So,
\begin{equation}\label{eqn-energy-1}
\begin{split}
&\frac{1}{16\pi}\int_{S}\iota^*(\omega^a{}_b-\ol\omega^a{}_b)\wedge N^0(\e_{ab02}\omega^2+\e_{ab03}\omega^3)\\
=&\frac{1}{8\pi}\int_{S}\|N^\perp\|\left(-\vv<\ol H,\ol X>+\vv<H,X>\right)dV_S.
\end{split}
\end{equation}
Furthermore,
\begin{equation}\label{eqn-energy-2}
\begin{split}
&\frac{1}{16\pi}\int_S\iota^*(\omega^a{}_b-\ol\omega^a{}_b)\wedge\eta^{bb}\e_{ab23}(-N^3\omega^2+N^2\omega^3)\\
=&\frac{1}{16\pi}\int_S\iota^*[\omega^0{}_1+\omega^1{}_0-(\ol\omega^0{}_1+\ol\omega^1{}_0)]\wedge(-N^3\omega^2+N^2\omega^3)\\
=&\frac{1}{8\pi}\int_S\iota^*(\omega^0{}_1-\ol\omega^0{}_1)\wedge(-N^3\omega^2+N^2\omega^3)\\
=&\frac{1}{8\pi}\int_S[(N^2\Gamma_{21}^0+N^3\Gamma_{31}^0)-(N^2\ol\Gamma_{21}^0+N^3\ol\Gamma_{31}^0)]\omega^2\wedge\omega^3\\
\end{split}
\end{equation}
where we have used the fact $\omega^0{}_1=\omega^1{}_0$ and $\iota^*\ol\omega^0{}_1=i^*\ol\omega^1{}_0$. Similarly as in \eqref{eqn-Gamma-22} and \eqref{eqn-o-Gamma-22}, it is not hard to see that
\begin{equation}
N^2\Gamma_{21}^0+N^3\Gamma_{31}^0=-\vv<\nabla_{N^2e_2+N^3e^3}e_1,e_0>=-\frac{1}{\|N^\perp\|}\vv<\nabla_{N^\top}X,N^\perp>
\end{equation}
and
\begin{equation}
N^2\ol\Gamma_{21}^0+N^3\ol\Gamma_{31}^0=-\vv<\ol\nabla_{N^2e_2+N^3e^3}e_1,e_0>_{\varphi^*\ol g}=-\frac{1}{\|N^\perp\|}\vv<\ol \nabla_{\ol N^\top}\ol X,\ol N^\perp>.
\end{equation}
Therefore,
\begin{equation}\label{eqn-energy-3}
\begin{split}
&\frac{1}{16\pi}\int_S\iota^*(\omega^a{}_b-\ol\omega^a{}_b)\wedge\eta^{bb}\e_{ab23}(-N^3\omega^2+N^2\omega^3)\\
=&\frac{1}{8\pi}\int_S\frac{1}{\|N^\perp\|}\left(\vv<\ol \nabla_{\ol N^\top}\ol  X,\ol N^\perp>-\vv<\nabla_{ N^\top}X,N^\perp>\right)dV_S.
\end{split}
\end{equation}
Combining \eqref{eqn-energy-1} and \eqref{eqn-energy-3}, we obtain \eqref{eqn-CNT} and hence Theorem \ref{thm-CNT-BY}.
\end{proof}

\section{Second variation of $E(x,y)$ at $x=y=0$}
In this section, we compute the second variation of the CNT quasi-local energy $E(x,y)$ for axially symmetric Kerr-like metrics at the obvious critical point $(x,y)=(0,0)$ as introduced in the first section. Throughout this section, we adopt the notations used there.

First, we have the following second variation of $\mathfrak{B}(x,y)$.
\begin{lem}\label{lem-2nd-B}
\begin{equation}\label{eqn-2nd-B}
\begin{split}
&\mathfrak{B}(\e Ru,\e \Sigma v)\\
=&-\frac{(\sqrt H\Sigma)_r}{ R}+\frac{\Sigma\left(1-\frac{H_{\theta\theta}}{2\Sigma^2}\right)}{\sqrt{1-\frac{H_\theta^2}{4H\Sigma^2}}}+\frac{H_\theta\Sigma_\theta}{2\Sigma^2\sqrt{1-\frac{H_\theta^2}{4H\Sigma^2}}}+\\
&\Bigg(-\frac{(\sqrt H\Sigma)_r}{2 R}(u^2+v^2)+\Sigma\left(\frac{1-\frac{1-\frac{H_{\theta\theta}}{2\Sigma^2}}{2\left(1-\frac{H_\theta^2}{4H\Sigma^2}\right)}}{\sqrt{1-\frac{H_\theta^2}{4H\Sigma^2}}}-\frac{H_\theta\Sigma_\theta}{4\Sigma^3\left(1-\frac{H_\theta^2}{4H\Sigma^2}\right)^{3/2}}\right)v^2\\
&+\sqrt H u_\theta v+\frac{H_\theta vv_\theta}{2\Sigma \sqrt{1-\frac{H_\theta^2}{4H\Sigma^2}}}\Bigg)\e^2+O(\e^3).
\end{split}
\end{equation}
\end{lem}
\begin{proof}
Suppose that $x=Rz$ and $y=\Sigma w$. Then,
\begin{equation}\label{eqn-a-b-l-2}
\left\{\begin{array}{l}\alpha=R\Sigma\sqrt{z^2+w^2+1}\\
\beta=2\sqrt H\Sigma\sqrt{w^2+1-\frac{H_\theta^2}{4H\Sigma^2}}\\
l=\Sigma^2(w^2+1).
\end{array}\right.
\end{equation}
Substituting these into \eqref{eqn-CNT-B}, we have
\begin{equation}
\begin{split}
&\frac{1}{\Sigma}\mathfrak{B}(x,y)\\
=&-\frac{(\sqrt H\Sigma)_r}{\Sigma R}\sqrt{z^2+w^2+1}+\frac{(w^2+1)-\frac{H_{\theta\theta}}{2\Sigma^2}}{\sqrt{w^2+1-\frac{H_\theta^2}{4H\Sigma^2}}}+\frac{H_\theta\Sigma_\theta}{2\Sigma^3\sqrt{w^2+1-\frac{H_\theta^2}{4H\Sigma^2}}}\\
&+\frac{\sqrt Hz_\theta w}{\Sigma\sqrt{z^2+w^2+1}}+\frac{H_\theta ww_\theta}{2\Sigma^2 (w^2+1)\sqrt{w^2+1-\frac{H_\theta^2}{4H\Sigma^2}}}-\frac{\sqrt Hzw^2w_\theta}{\Sigma(w^2+1)\sqrt{z^2+w^2+1}}.
\end{split}
\end{equation}
By setting $z=\e u$ and $w=\e v$ in the last identity and computing the Taylor expansion of each term with respect to $\e$ up to second order, we obtain the conclusion.
\end{proof}
Next, we compute the second variation of $E(x,y)$.
\begin{lem}\label{lem-2nd-E}
\begin{equation}\label{eqn-2nd-E}
\begin{split}
&E(\e Ru,\e\Sigma v)\\
=&\frac{1}{4}\int_0^\pi\left(-\frac{(\sqrt H\Sigma)_r}{ R}+\frac{\Sigma\left(1-\frac{H_{\theta\theta}}{2\Sigma^2}\right)}{\sqrt{1-\frac{H_\theta^2}{4H\Sigma^2}}}+\frac{H_\theta\Sigma_\theta}{2\Sigma^2\sqrt{1-\frac{H_\theta^2}{4H\Sigma^2}}}\right)d\theta+\\
&\frac{\e^2}{4}\int_0^\pi\left(-\frac{(\sqrt H\Sigma)_r}{2 R}(u^2+v^2)+\frac{\Sigma\sqrt{1-\frac{H_\theta^2}{4H\Sigma^2}}}{2} v^2+\sqrt H u_\theta v\right)d\theta+O(\e^3).
\end{split}
\end{equation}
\end{lem}
\begin{proof}
The conclusion is directly followed by taking the integration of \eqref{eqn-2nd-B}, applying integration by parts to the term
\begin{equation}
\int_0^\pi\frac{H_\theta vv_\theta}{2\Sigma \sqrt{1-\frac{H_\theta^2}{4H\Sigma^2}}}d\theta=\int_0^\pi\frac{H_\theta (v^2)_\theta}{4\Sigma \sqrt{1-\frac{H_\theta^2}{4H\Sigma^2}}}d\theta,
\end{equation}
and then simplify the coefficient of $v^2$.
\end{proof}
To take care of the term $\sqrt H u_\theta v$ in \eqref{eqn-2nd-E}, setting $u=-f(\theta)\cos\theta$ and $v=f(\theta)\sin\theta$ in \eqref{eqn-2nd-E}, and integrating by parts, we have the following
\begin{lem}\label{lem-2nd-E-f}
\begin{equation}\label{eqn-2nd-E-f}
\begin{split}
&E(-\e R f\cos\theta,\e Rf\sin\theta)\\
=&\frac{1}{4}\int_0^\pi\left(-\frac{(\sqrt H\Sigma)_r}{ R}+\frac{\Sigma\left(1-\frac{H_{\theta\theta}}{2\Sigma^2}\right)}{\sqrt{1-\frac{H_\theta^2}{4H\Sigma^2}}}+\frac{H_\theta\Sigma_\theta}{2\Sigma^2\sqrt{1-\frac{H_\theta^2}{4H\Sigma^2}}}\right)d\theta+\frac{\e^2}{4}\int_0^\pi Kf^2d\theta+O(\e^3)
\end{split}
\end{equation}
where
\begin{equation}\label{eqn-K}
K=-\frac{(\sqrt H\Sigma)_r}{2R}+\frac{\Sigma\sqrt{1-\frac{H_\theta^2}{4H\Sigma^2}}}{2}\sin^2\theta+\frac{H_\theta}{4\sqrt H}\cos\theta\sin\theta +\frac{\sqrt{H}}{2}.
\end{equation}
\end{lem}
\begin{proof}
Note that
\begin{equation}
\begin{split}
&\int_0^\pi \sqrt H u_\theta vd\theta\\
=&\int_0^\pi\sqrt{H}(-f_\theta\cos\theta+f\sin\theta)f\sin\theta d\theta\\
=&\int_0^\pi \sqrt H \sin^2\theta f^2d\theta-\frac{1}{2}\int_0^\pi \sqrt H\cos\theta\sin\theta (f^2)_\theta
d\theta\\
=&\frac{1}{2}\int_0^\pi \sqrt H f^2d\theta+\frac{1}{4}\int_0^\pi \frac{H_\theta}{\sqrt H}\cos\theta\sin\theta f^2d\theta\\
\end{split}
\end{equation}
Substituting this into \eqref{eqn-2nd-E}, we get the conclusion.
\end{proof}
By combing Lemma \ref{lem-2nd-E} and Lemma \ref{lem-2nd-E-f}, we have the following result.
\begin{thm}\label{thm-saddle}
If the mean curvature of $S$ along the radial outer normal is positive, and there is some point $\theta_0\in [0,\pi]$ such that $K(\theta_0)>0$, then $(x,y)=(0,0)$ is a saddle critical point of $E(x,y)$.
\end{thm}
\begin{proof}
By direct computation, the mean curvature of $S$ along the spacelike radial outer normal $\hat{\mathbf{r}}=(1/R)\partial_{r}$ is
\begin{equation}\label{k}
{k}=-(H^{-1}\vv<\nabla_{\partial_{\varphi}}\partial_{\varphi},\hat{\mathbf{r}}>
+\Sigma^{-2}\vv<\nabla_{\partial_{\theta}}\partial_{\theta},\hat{\mathbf{r}}>)=\frac{(\sqrt{H}\Sigma)_{r}}{\sqrt{H}R\Sigma}.
\end{equation}
So ${k}>0$ implies $(\sqrt{H}\Sigma)_{r}>0$. Then, by \eqref{eqn-2nd-E}, we know that
\begin{equation}
\delta^2E_{(0,0)}(R u,0)=-\frac{1}{2}\int_0^\pi\frac{(\sqrt H\Sigma)_r}{ R}u^2d\theta<0
\end{equation}
when $u$ is nonzero.

On the other hand, suppose that $K>0$ in a neighborhood $[a,b]$ of $\theta_0$, let $f$ be smooth function on $[0,\pi]$ such that $f= 1$ in a neighborhood of $\theta_0$ and $f=0$ outside $[a,b]$. Then, by \eqref{eqn-2nd-E-f},
 \begin{equation}
\delta^2E_{(0,0)}(-Rf\cos\theta ,\Sigma f\sin\theta)=\frac{1}{2}\int_0^\pi Kf^2d\theta>0.
\end{equation}
This completes the proof.
\end{proof}
By applying Theorem \ref{thm-saddle} to the Schwarschild and Kerr spacetimes, we have the following conclusion.
\begin{cor}
When the physical spacetime is the Schwarschild or Kerr spacetime, the obvious critical point $(x,y)=(0,0)$ of $E(x,y)$ is a saddle point.
\end{cor}
\begin{proof} For the Kerr spacetime,
\begin{equation}
\left\{\begin{array}{l}H\Sigma^2=\sin^2\theta((r^2+a^2)^2-\Delta a^2\sin^2\theta)\\
 R^2\Delta=\Sigma^2=r^2+a^2\cos^2\theta\\
 \Delta=r^2-2mr+a^2
\end{array}\right.
\end{equation}
with $0<a\leq m\leq r$. The Schwarschild spacetime corresponds to $a=0$ and $m>0$.

It is not hard to check that $(\sqrt H\Sigma)_r>0$ for Schwarschild and Kerr spacetimes. On the other hand, by direct computation, we have
\begin{equation}
H_\theta\left(\frac{\pi}{2}\right)=0
\end{equation}
and
\begin{equation}
\begin{split}
&\Sigma(\pi/2)=r,\ \sqrt H(\pi/2)=r\left(1+\frac{a^2}{r^2}+\frac{2ma^2}{r^3}\right)^{1/2},\\
&R(\pi/2)=\left(1-\frac{2m}{r}+\frac{a^2}{r^2}\right)^{-1/2}.
\end{split}
\end{equation}
Substituting these into \eqref{eqn-K}, we have
\begin{equation}
\begin{split}
&K(\pi/2)\\
\geq&r\left(-\left(1+\frac{a^2}{r^2}+\frac{2ma^2}{r^3}\right)^{1/2}\left(1-\frac{2m}{r}+\frac{a^2}{r^2}\right)^{1/2}+\frac{1+\left(1+\frac{a^2}{r^2}+\frac{2ma^2}{r^3}\right)^{1/2}}{2}\right)\\
\geq &r\left(-\left(1+\frac{a^2}{r^2}+\frac{2m}{r}\right)^{1/2}\left(1-\frac{2m}{r}+\frac{a^2}{r^2}\right)^{1/2}+1\right)\\
=&r\left(-\left(\left(1+\frac{a^2}{r^2}\right)^2-\frac{4m^2}{r^2}\right)^{1/2}+1\right)\\
\geq&r\left(-\left(1-\frac{m^2}{r^2}\right)^{1/2}+1\right)\\
>&0
\end{split}
\end{equation}
when $m>0$. Then, the conclusion follows by Theorem \ref{thm-saddle}.
\end{proof}
For the Minkowski spacetime, since $K\equiv 0$, we have to deal with it independently.
\begin{cor}
When the physical spacetime is the Minkowski spacetime, the obvious critical point $(x,y)=(0,0)$ of $E(x,y)$ is a saddle point.
\end{cor}
\begin{proof}
For the Minkowski spacetime, we have $H=r^2\sin^2\theta$, $\Sigma=r$ and $R=1$. By Lemma \ref{lem-2nd-E},
\begin{equation}
\delta^2{E}_{(0,0)}(u,rv)=\frac{r}{2}\int_0^\pi \left(-u^2-\frac{1}{2}v^2+u_\theta v\right)\sin\theta d\theta.
\end{equation}
It is clear that
\begin{equation}
\delta^2E_{(0,0)}(1,0)<0.
\end{equation}
On the other hand, by direct computation,
\begin{equation}
\delta^2E_{(0,0)}\left(-\frac{2}{3}\cos^2\theta,r\sin\theta\cos\theta\right)=\frac{r}{45}>0.
\end{equation}
This completes the proof.
\end{proof}

\section{Appendix: The choice $N=N_{0}$ and critical equation}
Consider the displacement $N$ is $N_{0}$:
\begin{equation}
N=N_{0}=\sqrt{1+\parallel\nabla\tau\parallel^2}e_{0}-\nabla\tau
\end{equation}
where $N^{\top}=-\nabla\tau$ is tangent to the $S$ with metric $\sigma=\Sigma^2d\theta^2+Hd\phi^2$, so that
\begin{equation}
-\nabla\tau=-\frac{\tau_{\theta}}{\Sigma^2}\partial_{\theta}-\frac{\tau_{\varphi}}{H}\partial_{\varphi}=-\frac{\tau_{\theta}}{\Sigma^2}\partial_{\theta}
\end{equation}
for the choice $\tau$ depends only on $\theta$ as the previous setting $y=\tau_{\theta}$. The Laplacian of $\tau$ w.r.t.\ $\sigma$ is
\begin{equation}
\begin{split}
\Delta\tau&=\sigma^{AB}\nabla_{A}\nabla_{B}\tau=\sigma^{AB}(\partial_{A}\partial_{B}\tau-\Gamma^{C}{}_{AB}\partial_{C}\tau)\nonumber\\
&=\frac{1}{\Sigma^{2}}\tau_{\theta\theta}+\frac{\Sigma H_{\theta}-2H\Sigma_{\theta}}{2H\Sigma^3}\tau_{\theta}\nonumber\\
&=\frac{y_{\theta}}{\Sigma^2}+\frac{\Sigma H_{\theta}-2H\Sigma_{\theta}}{2H\Sigma^3}y.
\end{split}
\end{equation}
Recall the displacement vector $N=N^{\mu}\partial_{\mu}$ determined by 4D isometric matching with the components \cite{SCLN2}
\begin{equation}
N^t=\frac{\sqrt{H}\alpha}{\sqrt{-g}},\quad N^r=-\frac{x}{R^2},\quad N^\theta=-\frac{y}{\Sigma^2},\quad N^\varphi=\frac{-G\alpha}{\sqrt{-g}\sqrt{H}},
\end{equation}
read for $N^{\top}=N^{\theta}\partial_{\theta}$. According to the spacelike mean curvature vector of $S$ in Kerr spacetime is $h=-(k/R)\partial_{r}$, from \eqref{eqn-e-0} and \eqref{k} we have
\begin{equation}
\sqrt{1+\Vert\nabla\tau\Vert^2}\langle h,e_{0}\rangle=\langle h,N\rangle=\frac{k}{R}x=-\Delta\tau
\end{equation}
which is the first equation of \eqref{eqn-critical}. The choice $N=N_{0}$ satisfies the critical equation that means the solution $x$ is determined by $y(\theta)$ which is the only free choice for $N=N_{0}$.

\end{document}